\documentclass[12pt]{amsart}
\makeatletter
\@namedef{subjclassname@2020}{%
  \textup{2020} Mathematics Subject Classification}
\makeatother

\usepackage{mathpazo}
\usepackage{avant}
\usepackage{amssymb}
\usepackage{mathtools}
\usepackage{xurl}
\usepackage[backend=biber,style=alphabetic]{biblatex}

\newtheorem{thm}{Theorem}[section]
\newtheorem{cor}[thm]{Corollary}
\newtheorem{lem}[thm]{Lemma}
\newtheorem{prop}[thm]{Proposition}
\theoremstyle{definition}
\newtheorem{defn}[thm]{Definition}
\newtheorem{example}[thm]{Example}
\theoremstyle{remark}
\newtheorem{rem}[thm]{Remark}
\numberwithin{equation}{section}

\newcommand{\CA}{\mathcal A}
\newcommand{\CO}{\mathcal O}
\newcommand{\C}{\mathbb C}
\newcommand{\R}{\mathbb R}
\newcommand{\cinf}{\mathcal C^\infty}
\newcommand{\delb}{\overline{\partial}}
\newcommand{\MC}{\operatorname{MC}}
\newcommand{\End}{\operatorname{End}}

\newcommand{\Span}{\operatorname{span}}
\newcommand{\rhom}{\operatorname{\mathbf{R}Hom}}
\newcommand{\ext}{\operatorname{Ext}}

\begin{document}

\title[Superconnections on transversely holomorphic foliations]
{Superconnections, descent, and monodromy on transversely holomorphic
foliations}

\author{Qingyun Zeng}
\address{Department of Mathematics, University of Pennsylvania,
Philadelphia, PA 19104}
\email{qze@math.upenn.edu}


\begin{abstract}
Let \(X\) carry a transversely holomorphic foliation, equivalently an
elliptic involutive structure \(V\subset T_{\C}X\), and let \(\CO_V\) be its
sheaf of leafwise-constant, transversely holomorphic functions.  We construct
a finite superconnection model for the derived category of coherent
\(\CO_V\)-modules.  The key input is a mixed local reduction for finite
Maurer--Cartan objects over the mixed de Rham--Dolbeault dga
\((\wedge^\bullet V^\vee,d_V)\): a multiplicative homotopy contracts the
real directions, after which Block's Dolbeault gauge theorem removes the
positive transverse form degrees.  For compact \(X\), this gives an exact
equivalence between the homotopy category of bounded finite-rank flat
\(V\)-superconnections and \(D^b_{\mathrm{coh}}(X,\CO_V)\), interpolating
between the de Rham and Dolbeault realizations.  We prove locally finite
\v{C}ech descent under necessary uniform amplitude and rank bounds, identify
the coherent heart with equivariant coherent analytic sheaves on a
transverse monodromy groupoid, and characterize descent to ordinary
holonomy.  For a holomorphic suspension, we identify the full
superconnection category, up to Morita equivalence, with the homotopy fixed
points of the Dolbeault category of the transversal and derive an
equivariant Ext spectral sequence.  Examples on \(S^1\) and \(S^2\) delimit
when ordinary monodromy \(1\)-groupoids can recover the derived category.
\end{abstract}

\maketitle

\section{Introduction}

A transversely holomorphic foliation interpolates between a smooth manifold
and a complex manifold: its local geometry is a product of real plaques and
complex transversals.  Its natural structure sheaf \(\CO_V\) consists of
functions constant on the plaques and holomorphic in the transverse
directions.  We construct a finite differential-geometric model for
\(D^b_{\mathrm{coh}}(X,\CO_V)\) which simultaneously recovers the de Rham
model for finite homotopy local systems and the Dolbeault model for coherent
analytic complexes.

The objects in the model are bounded graded finite-rank bundles equipped
with flat superconnections along the elliptic involutive structure \(V\).
They may be viewed as representations up to homotopy of the complex Lie
algebroid \(V\), in the sense of the general theory developed in
\cite{AriasAbadCrainic2012}; in Block's terminology they are cohesive
modules.  For a compact complex manifold, Block identified their Dolbeault
homotopy category with the bounded derived category of coherent analytic
sheaves \cite[Theorem~4.3]{Block2010}.  The central question is whether this
equivalence survives for the mixed de Rham--Dolbeault dga of a transversely
holomorphic foliation.

Korman proved the local normal form for an EIS, its Poincar\'e lemma, and the
equivalence between ordinary flat \(V\)-bundles and locally free
\(\CO_V\)-modules \cite[Theorems~5--7]{Korman2014}.  Those results do not by
themselves treat the higher components of a cohesive superconnection.
Chuang--Holstein--Lazarev give a general local-to-global comparison for a fine
sheaf of dgas \cite[Proposition~7.4, Theorem~7.10 and
Corollary~7.13]{CHL2021}.  The nonformal point needed here is their local
Maurer--Cartan condition.

The first main result, Theorem~\ref{thm:local-mc}, supplies this condition.
On an adapted chart \(B^d\times\Delta^n\), every finite Maurer--Cartan object
over
\[
\Omega^\bullet(B^d)\widehat\otimes
\CA^{0,\bullet}(\Delta^n)
\]
becomes, after shrinking, homotopy gauge equivalent to a finite complex over
\(\CO(\Delta^n)\).  A multiplicative contraction removes the real
directions, and Block's same-bundle Dolbeault gauge construction removes the
remaining positive form degrees.  This is stronger than the ordinary
Poincar\'e lemma, which does not control the higher superconnection
components.

After verifying the remaining sheaf-dga hypotheses and proving that the
finite form of the Chuang--Holstein--Lazarev condition is sufficient, we
obtain
\[
H^0\!\left(\mathcal P_{\CA^\bullet(X)}\right)
\simeq
D_{\mathrm{perf}}^B(X,\CO_V)
\]
for every \(X\).  If \(X\) is compact, this becomes the realization theorem
\[
H^0\!\left(\mathcal P_{\CA^\bullet(X)}\right)
\simeq
D^b_{\mathrm{coh}}(X,\CO_V).
\]
Compactness is used only in the second identification.  An earlier version
of the compact statement was announced in the author's thesis
\cite[Theorem~IV and Proposition~22.19]{ZengThesis}; the present argument
replaces its unproved flatness step by the mixed local reduction and an
independent verification of the comparison hypotheses.

The superconnection model also satisfies locally finite descent.  Adapting
Wei's partition-of-unity and quasi-representability argument
\cite{Wei2023}, we prove a dg quasi-equivalence with globally bounded twisted
objects.  For an infinite cover the common amplitude and rank bounds are
essential; Example~\ref{ex:unbounded-descent} gives an explicit obstruction
to unrestricted descent.  For a finite cover, and hence for the cover used
in the suspension calculation below, the bounded and unrestricted
categories coincide.

At the abelian level, restriction to a complete transversal identifies
\(\operatorname{Coh}(X,\CO_V)\) with equivariant coherent analytic sheaves
on the transverse monodromy groupoid.  This description is invariant under
changing the complete transversal and hence defines a coherent category on
a ringed analytic monodromy stack.  Passage to ordinary holonomy retains
exactly the objects on which the monodromy-to-holonomy kernel acts strictly
trivially.  The \(S^1\) and \(S^2\) examples show why this heart-level
statement cannot be promoted to a general derived equivalence with an
ordinary \(1\)-groupoid.

The strongest positive derived result is obtained for a holomorphic
suspension \(X_\phi\) of a biholomorphism \(\phi:Z\to Z\).  If
\(\mathcal C_Z=\mathcal P_{\CA^{0,\bullet}(Z)}\), then
\[
\mathcal P_{\CA^\bullet(X_\phi)}
\simeq
\mathcal C_Z^{h\phi}
\]
in \(\operatorname{Cat}^{\mathrm{perf}}_{\infty,\C}\), equivalently up to
Morita equivalence.  For equivariant coherent sheaves this gives
\[
H^p\!\left(\mathbb Z,
\operatorname{Ext}_Z^q(\mathcal F,\mathcal G)\right)
\Longrightarrow
\operatorname{Ext}_{\CO_V}^{p+q}
(\widetilde{\mathcal F},\widetilde{\mathcal G}).
\]
The elliptic-curve suspension by \([-1]\) exhibits a non-product case in
which monodromy removes transverse Ext classes.

Thus the general comparison and twisted-descent formalisms are imported
from Chuang--Holstein--Lazarev and Wei.  The new content is the finite mixed
local reduction, the exact verification needed to apply those formalisms to
an EIS, the sharp bounded descent statement, the analytic monodromy
interpretation of the coherent heart, and the derived suspension and Ext
calculation.

A companion work develops a higher Riemann--Hilbert correspondence and
descent theorem for regular smooth foliations
\cite{ZengHigherRH}.  Its source is the leafwise de Rham dga and its target
consists of smooth infinity-local systems on the foliated monodromy
infinity-groupoid.  It does not include the transverse Dolbeault
coefficients needed for the mixed EIS comparison studied here, and no result
from it is used in the proofs below.

A separate in-preparation paper applies related finite-superconnection
methods to a categorical Penrose--Ward correspondence for self-dual
Yang--Mills theory \cite{ZengYangMills}; it neither supplies nor is used in
the present results.

Further companion manuscripts treat singular-foliation directions:
characteristic classes of \(L_\infty\)-resolved singular foliations
\cite{BlockZengCharacteristicClasses}, global Dolbeault
\(L_\infty\)-algebroid resolutions of singular holomorphic foliations
\cite{BlockWeiZengDolbeaultResolutions}, and higher monodromy and higher
transverse holonomy models
\cite{ZengHigherMonodromy,ZengHigherHolonomy}.  These works concern singular
foliations rather than the regular EIS considered here, and no statement
from them is used below.

The paper is organized as follows.  We first review elliptic involutive
structures and isolate the finite sheafification problem.  We then prove the
mixed local reduction and the global realization and descent theorems.
After describing the coherent heart by monodromy and holonomy groupoids, we
examine the basic de Rham, Dolbeault, and product cases and finish with the
derived suspension theorem and its equivariant Ext calculation.

\section*{Acknowledgments}

This article substantially refines part of the author's Ph.D.\ thesis
\cite{ZengThesis} and forms part of a broader program applying
higher-categorical methods to foliations and related geometry
\cite{ZengHigherRH,BlockZengCharacteristicClasses,ZengYangMills,
BlockWeiZengDolbeaultResolutions}.

The author is deeply grateful to Jonathan Block for his guidance,
illuminating discussions, encouragement, and support, and thanks Jim
Stasheff, Sylvain Lavau, Zhaoting Wei, and Yingdi Qin for helpful discussions.

\section{Elliptic involutive structures}

Let \(X\) be a smooth manifold without boundary.  We use the convention that
smooth manifolds are Hausdorff, second countable, and finite-dimensional.

\begin{defn}
An \emph{elliptic involutive structure} on \(X\) is a complex subbundle
\[
V\subset T_{\C}X
\]
such that
\[
[\Gamma(V),\Gamma(V)]\subset\Gamma(V),
\qquad
V+\overline V=T_{\C}X.
\]
We write
\[
\CA_X^\bullet=(\wedge^\bullet V^\vee,d_V)
\]
for its Chevalley--Eilenberg dga and
\[
\CO_V=\ker(d_V:\CA_X^0\longrightarrow\CA_X^1)
\]
for its structure sheaf.
\end{defn}

The symbol sequence of \(d_V\) is the Koszul complex of the dual anchor.
Thus the second condition in the definition is equivalent to ellipticity of
\((\CA_X^\bullet,d_V)\).  This agrees with Block's definition for a complex
Lie algebroid; compare \cite[Section~5.2]{Block2010}.

\begin{thm}[Local normal form]\label{thm:eis-normal-form}
For every \(x\in X\), there is an adapted chart
\[
U=B^d\times\Delta^n
\]
with real coordinates \(t_1,\ldots,t_d\) and complex coordinates
\(z_1,\ldots,z_n\) in which
\[
V|_U=
\Span_{\C}\left\{
\frac{\partial}{\partial t_1},\ldots,
\frac{\partial}{\partial t_d},
\frac{\partial}{\partial\overline z_1},\ldots,
\frac{\partial}{\partial\overline z_n}
\right\}
=\Span_{\C}\{dz_1,\ldots,dz_n\}^{\perp}.
\]
In particular, \(\dim_{\R}X=d+2n\).
\end{thm}

\begin{proof}
This is the Newlander--Nirenberg theorem for elliptic involutive structures
\cite[Theorem~5]{Korman2014}; see also
\cite[Chapter~VIII]{BCH2008} and \cite{Treves1992}.
\end{proof}

Let \(D_{\R}\subset TX\) be the real distribution characterized by
\[
D_{\R}\otimes_{\R}\C=V\cap\overline V.
\]
The local normal form shows that \(D_{\R}\) has constant rank \(d\) and is
involutive.  Thus an EIS is equivalently a transversely holomorphic
foliation.

\begin{prop}\label{prop:local-ring}
On an adapted chart \(U=B^d\times\Delta^n\), with projection
\(\pi:U\to\Delta^n\),
\[
\CO_V|_U\simeq\pi^{-1}\CO_{\Delta^n}.
\]
Consequently,
\[
(\CO_V)_x\simeq\CO_{\C^n,0}
\]
for every \(x\in U\).  The sheaf \(\CO_V\) is coherent, and its stalks are
regular local rings of dimension \(n\).
\end{prop}

\begin{proof}
In the coordinates of Theorem~\ref{thm:eis-normal-form}, the equation
\(d_Vf=0\) is
\[
\frac{\partial f}{\partial t_a}=0,
\qquad
\frac{\partial f}{\partial\overline z_j}=0.
\]
Hence \(f\) is independent of \(t\) and holomorphic in \(z\), which gives the
first two assertions.  To check coherence, shrink to a product
\(B'\times\Delta'\) with \(B'\) connected.  Sections of
\(\pi^{-1}\CO_{\Delta'}\) on this product are precisely
\(\CO(\Delta')\).  Hence the kernel of a map between finite free
\(\pi^{-1}\CO_{\Delta'}\)-modules is the inverse image of the corresponding
analytic kernel on \(\Delta'\), which is locally finitely generated by Oka's
coherence theorem.  Thus \(\pi^{-1}\CO_{\Delta'}\), and hence \(\CO_V\), is
coherent.  Finally, the convergent power-series ring
\(\CO_{\C^n,0}\) is regular of dimension \(n\); see
\cite{GrauertRemmert1984}.
\end{proof}

\begin{cor}\label{cor:coherent-resolution}
Every coherent \(\CO_V\)-module has, locally, a finite free resolution of
projective dimension at most \(n\), equivalently a resolution with at most
\(n+1\) nonzero free terms.
\end{cor}

\begin{proof}
This is the corresponding statement over the regular local ring
\(\CO_{\C^n,0}\), transported through Proposition~\ref{prop:local-ring}.
Matrices and their relations are represented after shrinking the adapted
chart.
\end{proof}

\begin{rem}
Korman's Theorem~7 identifies locally free \(\CO_V\)-modules with smooth
vector bundles carrying a flat \(V\)-connection \cite{Korman2014}.  This is
the degree-zero case of the local result proved in Section~\ref{sec:local-mc}.
\end{rem}

\section{Finite superconnections and sheafification}

Set
\[
\CA^\bullet=\Gamma(X,\CA_X^\bullet),
\qquad
\CA^0=\cinf(X).
\]
A \emph{cohesive module} over \(\CA^\bullet\) is a bounded graded,
finitely generated projective right \(\CA^0\)-module \(E^\bullet\), together
with a degree-one \(\mathbb Z\)-connection
\[
\mathbb E:
E^\bullet\otimes_{\CA^0}\CA^\bullet
\longrightarrow
E^\bullet\otimes_{\CA^0}\CA^\bullet
\]
which satisfies the Leibniz rule and \(\mathbb E^2=0\).  We denote Block's
dg category of cohesive modules by \(\mathcal P_{\CA^\bullet}\); see
\cite[Section~2]{Block2010}.

The sheafification of \(E^p\) is
\[
\widetilde E^p
=\CA_X^0\otimes_{\CA^0}E^p.
\]
Extending the \(\mathbb Z\)-connection defines a dg
\(\CA_X^\bullet\)-module
\[
\Phi(E,\mathbb E)
=
\left(
\CA_X^\bullet\otimes_{\CA_X^0}\widetilde E^\bullet,
\mathbb E
\right).
\]
Restriction of scalars along the quasi-isomorphism
\(\CO_V\to\CA_X^\bullet\) then regards \(\Phi(E,\mathbb E)\) as an object of
the derived category of \(\CO_V\)-modules.  Notice that the first tensor
product is over \(\CA^0\), not over \(\CA^\bullet\).

We use \(D_{\mathrm{perf}}^B(X,\CO_V)\) for the derived category of globally
bounded perfect dg sheaves in the sense of
\cite[Definition~7.3 and the paragraph following Remark~7.2]{CHL2021}.

\begin{lem}[Finite smooth generators]
\label{lem:finite-smooth-generators}
Let \(M\) be a finite-dimensional paracompact smooth manifold and let
\(E\to M\) be a smooth complex vector bundle of constant finite rank.  Then
\(E\) is a direct summand of a finite-rank trivial bundle, and
\(\Gamma(M,E)\) is a finitely generated projective
\(\cinf(M)\)-module.
\end{lem}

\begin{proof}
Put \(d=\dim M\) and \(r=\operatorname{rank}E\).  The space \(M\) is
metrizable and has covering dimension \(d\).  By Ostrand's discrete-refinement
theorem \cite[Theorem~1]{Ostrand1965}, a trivializing cover has a refinement
\(\{W_{a\lambda}\}_{0\leq a\leq d,\ \lambda\in\Lambda_a}\)
such that each family \(\{W_{a\lambda}\}_{\lambda\in\Lambda_a}\) is
discrete.  Choose a subordinate smooth partition of unity
\(\{\rho_{a\lambda}\}\) and local frames
\(\{e_{a\lambda,j}\}_{j=1}^r\).  Since
\(\operatorname{supp}\rho_{a\lambda}\subset W_{a\lambda}\), the products
\(\rho_{a\lambda}e_{a\lambda,j}\) extend smoothly by zero.  The sums
\(s_{a,j}:=\sum_\lambda\rho_{a\lambda}e_{a\lambda,j}\) are smooth because
each is locally either zero or one term.  At every point some
\(\rho_{a\lambda}\) is nonzero, so the corresponding \(r\) sections span
the fiber.  Thus the \(s_{a,j}\) define a surjection
\(M\times\C^{(d+1)r}\to E\).  Its kernel is a smooth subbundle, which a
Hermitian metric splits.  Global sections then exhibit \(\Gamma(M,E)\) as a
direct summand of \(\cinf(M)^{(d+1)r}\).
\end{proof}

Chuang--Holstein--Lazarev state their condition \({(*)}\) for every free
graded \(\C\)-module \(G\).  We write \({(*)}_{\mathrm{fg}}\) for the same
condition restricted to finitely generated free graded modules.  Such a
module is finite-dimensional and bounded.

\begin{lem}\label{lem:finite-condition}
In the proof of \cite[Proposition~7.4 and Theorem~7.10]{CHL2021} for
perfect twisted modules, it is enough to impose their condition \({(*)}\)
for finite-dimensional graded vector spaces \(G\).
\end{lem}

\begin{proof}
The proof of \cite[Proposition~7.4]{CHL2021} first treats finitely generated
twisted modules.  Such a module has underlying graded module
\(G\otimes\CA^\bullet\), where \(G\) is a finitely generated free graded
\(\C\)-module \cite[Definition~3.1]{CHL2021}.  A finite homogeneous
generating set makes \(G\) finite-dimensional and supported in only finitely
many degrees.

Condition \({(*)}\) enters that proof only twice.  In
\cite[Lemma~7.5]{CHL2021}, it is applied to the finite \(G\) underlying one
finitely generated twisted module to prove local perfection.  In
\cite[Lemma~7.6]{CHL2021}, it is applied to the finite graded modules
underlying two such objects to compare their mapping complexes.  The passage
from finitely generated to perfect twisted modules is then made by homotopy
idempotent completion and introduces no new Maurer--Cartan element.
Finally, \cite[Theorem~7.10]{CHL2021} combines Proposition~7.4 with essential
surjectivity from Lemma~7.8; the latter uses strictification of perfect
sheaves and Block's quasi-representability theorem, but not condition
\({(*)}\).  Thus every use of \({(*)}\) in this comparison has
finite-dimensional \(G\).
\end{proof}

\begin{rem}\label{rem:infinite-G}
We neither prove nor use condition \({(*)}\) for arbitrary infinite-rank or
unbounded free graded \(G\).  Lemma~\ref{lem:finite-condition} isolates the
finite form required for the perfect categories appearing in this paper, so
none of our statements implies the unrestricted condition.
\end{rem}

\section{The local Maurer--Cartan theorem}\label{sec:local-mc}

For a dga \(A\) and a finite-dimensional graded vector space \(G\), write
\(\MC_{\mathrm{dg}}(A\otimes\End(G))\) for the Maurer--Cartan dg category of
\cite[Section~2]{CHL2021}.  Isomorphism in its homotopy category is called
\emph{homotopy gauge equivalence}.

\begin{thm}[Local Maurer--Cartan reduction]\label{thm:local-mc}
Let \(x\in X\), let \(U'\) be a neighborhood of \(x\), and let \(G\) be a
finite-dimensional graded complex vector space.  For every
\[
\xi\in
\MC\!\left(
\CA_X^\bullet(U')\widehat\otimes\End(G)
\right)
\]
there are an adapted product neighborhood
\[
x\in U=B^d\times\Delta^n\subset U'
\]
and
\[
\eta\in
\MC\!\left(
\CO_V(U)\otimes\End(G)
\right)
\]
such that \(\xi|_U\) and the image of \(\eta\) are homotopy gauge equivalent
in
\[
\MC_{\mathrm{dg}}\!\left(
\CA_X^\bullet(U)\widehat\otimes\End(G)
\right).
\]
Thus \(\CA_X^\bullet\) satisfies condition
\({(*)}_{\mathrm{fg}}\).
\end{thm}

\begin{proof}
Choose an adapted product chart
\(U_0=B^d\times\Delta^n\subset U'\) centered at \(x=(t_0,z_0)\).  As a
topological dga,
\[
\CA_X^\bullet(U_0)
\simeq
\Omega^\bullet(B^d)\widehat\otimes
\CA^{0,\bullet}(\Delta^n),
\qquad
d_V=d_{B^d}+\delb.
\]
These spaces are nuclear Fr\'echet spaces, so the completed projective tensor
product has the stated product identification
\cite[Theorem~51.6]{Treves1967}; compare
\cite[Theorem~A.6]{CHL2021}.  Tensoring with the finite-dimensional algebra
\(\End(G)\) introduces no further completion issue.

The dga \(\CA_X^\bullet(U_0)\) is the quotient of the complexified de Rham
dga \(\Omega^\bullet_{\C}(U_0)\) by the closed differential ideal generated
by \(V^\perp\).  Chuang--Holstein--Lazarev regard de Rham algebras as dg
Arens--Michael algebras \cite[Section~4 and Section~8.1]{CHL2021}; closed
quotients and finite-dimensional tensor products preserve that property.
Thus all dgas in the following homotopy argument lie in the topological
category required by their Proposition~4.6.

Let \(p:U_0\to\Delta^n\) be projection and
\(j:\Delta^n\to U_0\) the section \(j(z)=(t_0,z)\).  Pullback gives continuous
dga morphisms
\[
p^*:\CA^{0,\bullet}(\Delta^n)\longrightarrow
\CA_X^\bullet(U_0),
\qquad
j^*:\CA_X^\bullet(U_0)\longrightarrow
\CA^{0,\bullet}(\Delta^n),
\]
with \(j^*p^*=1\).  The contraction
\[
H_s(t,z)=((1-s)t+s t_0,z)
\]
gives a smooth homotopy through dga maps from \(1\) to \(p^*j^*\).
Indeed, pullback along
\((t,z,s)\mapsto(H_s(t,z),s)\) is a continuous multiplicative map to the
completed tensor product with \(\Omega^\bullet([0,1])\), so it is a smooth
homotopy in the precise sense of \cite[Section~4]{CHL2021}.
After tensoring with \(\End(G)\), \cite[Proposition~4.6]{CHL2021} implies
that \(\xi|_{U_0}\) is homotopy gauge equivalent to
\(p^*j^*(\xi|_{U_0})\).  This uses the multiplicative pullback homotopy,
not merely the linear chain contraction in the Poincar\'e lemma.

The element \(j^*(\xi|_{U_0})\) is a finite Maurer--Cartan object over the
Dolbeault dga of \(\Delta^n\).  Block's local Dolbeault gauge theorem
\cite[Lemma~4.5]{Block2010}, equivalently the local condition used in
\cite[Theorem~8.3]{CHL2021}, gives a smaller polydisc
\(\Delta'\ni z_0\) and
\[
\eta\in
\MC\!\left(\CO(\Delta')\otimes\End(G)\right)
\]
whose image is homotopy gauge equivalent to
\(j^*(\xi|_{U_0})|_{\Delta'}\).  Pulling this equivalence back by \(p\) and
using
\(\CO_V(B^d\times\Delta')=\CO(\Delta')\) proves the result with
\(U=B^d\times\Delta'\).
\end{proof}

\begin{rem}\label{rem:boundary-cases}
If \(d=0\), the contraction step is the identity and the proof is Block's
Dolbeault argument.  If \(n=0\), the Dolbeault step is vacuous and the proof
is the de Rham homotopy-invariance argument of
\cite[Theorem~8.1]{CHL2021}.  If \(G\) is concentrated in degree zero, the
statement reduces to the local parallel-frame theorem for a flat
\(V\)-bundle \cite[Theorem~7]{Korman2014}.
\end{rem}

\section{The global comparison}

\begin{prop}\label{prop:chl-hypotheses}
The morphism of sheaves of dgas
\[
\CO_V\longrightarrow\CA_X^\bullet
\]
satisfies the hypotheses of
\cite[Corollary~3.4, Proposition~7.4 and Theorem~7.10]{CHL2021}, with the
finite form of condition \({(*)}\) from
Lemma~\ref{lem:finite-condition}.
\end{prop}

\begin{proof}
The local Poincar\'e lemma for an EIS makes
\(\CO_V\to\CA_X^\bullet\) a quasi-isomorphism; see
\cite[Theorem~6]{Korman2014}.  More explicitly, on
\(B^d\times\Delta^n\) it is the tensor product of the de Rham resolution on
the ball and the Dolbeault resolution on the polydisc.

Every \(\CA_X^q\) is the sheaf of smooth sections of the finite-rank bundle
\(\wedge^qV^\vee\), hence is fine.  The ringed space
\((X,\CA_X^0)=(X,\cinf_X)\) is locally ringed and commutative, and each
\(\CA_X^q\) is locally free of finite rank, hence flat, over \(\CA_X^0\).
A smooth manifold is paracompact, Hausdorff, and has finite covering
dimension.

The global dga is concentrated in the finite range
\(0\leq q\leq\operatorname{rank}_{\C}V\).  For every \(q\), the finite-rank
bundle \(\wedge^qV^\vee\) satisfies
Lemma~\ref{lem:finite-smooth-generators}, so
\[
\CA^q(X)=\Gamma(X,\wedge^qV^\vee)
\]
is a finitely generated projective \(\cinf(X)\)-module.  Consequently
\(\CA^\#(X)\) is flat over \(\CA^0(X)=\cinf(X)\).  These two facts are also
the hypotheses needed to pass between perfect twisted and cohesive modules
in \cite[Corollary~3.4]{CHL2021}.

It remains to check local goodness.  On an adapted
\(U=B^d\times\Delta^n\), the fine resolution \(\CA_X^\bullet|_U\) computes
\(R\Gamma(U,\CO_V)\).  Its global-section complex is
\[
\Omega^\bullet(B^d)\widehat\otimes\CA^{0,\bullet}(\Delta^n).
\]
The de Rham contraction of \(B^d\) is a homotopy equivalence from this
complex to the Dolbeault complex of \(\Delta^n\).  Cartan's theorem~B on the
Stein polydisc then gives
\[
H^q(U,\CO_V)=0\quad(q>0),
\qquad
H^0(U,\CO_V)=\CO_V(U),
\]
so adapted product charts form a basis of good neighborhoods.  Finally,
Theorem~\ref{thm:local-mc} supplies the required finite Maurer--Cartan
condition.
\end{proof}

\begin{thm}\label{thm:noncompact}
For every smooth manifold \(X\) with an elliptic involutive structure,
sheafification induces an equivalence
\[
H^0\!\left(\mathcal P_{\CA^\bullet}\right)
\simeq
D_{\mathrm{perf}}^B(X,\CO_V).
\]
\end{thm}

\begin{proof}
By Proposition~\ref{prop:chl-hypotheses}, Theorem~7.10 and
Corollary~7.13 of \cite{CHL2021} apply.  Their comparison between perfect
twisted modules and perfect dg sheaves transfers to Block's cohesive modules
through \cite[Corollary~3.4]{CHL2021}.
\end{proof}

\begin{rem}[Dg enhancement]
\label{rem:dg-enhancement}
After composing the twisted-module sheafification functor with functorial
cofibrant replacement, \cite[Remark~7.5]{CHL2021} enhances the preceding
equivalence to a quasi-equivalence with the dg category of fibrant-cofibrant
perfect dg sheaves.  This does not assert that the unmodified
sheafification functor is quasi-fully faithful; that stronger conclusion is
the projective case distinguished in the same remark.
\end{rem}

\begin{rem}
For noncompact \(X\), the global-boundedness condition in
\(D_{\mathrm{perf}}^B(X,\CO_V)\) is essential.  We make no identification
there with all bounded complexes having coherent cohomology.
\end{rem}

\begin{prop}\label{prop:perf-coh}
If \(X\) is compact, then
\[
D_{\mathrm{perf}}^B(X,\CO_V)
=D^b_{\mathrm{coh}}(X,\CO_V).
\]
\end{prop}

\begin{proof}
By Proposition~\ref{prop:local-ring}, \(\CO_V\) is coherent and all its
stalks are regular local rings of dimension \(n\).  Hence every bounded
complex with coherent cohomology is locally perfect, with projective
amplitude enlarged by at most \(n\).  Conversely, a perfect complex has
coherent cohomology because \(\CO_V\) is coherent.  A finite adapted cover of
compact \(X\) gives uniform amplitude and rank bounds, which is precisely
global boundedness.
\end{proof}

\begin{thm}[Superconnection realization theorem]\label{thm:main}
Let \(X\) be a compact smooth manifold without boundary and let
\(V\subset T_{\C}X\) be an elliptic involutive structure.  Then
sheafification induces an exact equivalence
\[
\boxed{
H^0\!\left(\mathcal P_{\CA^\bullet}\right)
\simeq
D^b_{\mathrm{coh}}(X,\CO_V).
}
\]
\end{thm}

\begin{proof}
Combine Theorem~\ref{thm:noncompact} with
Proposition~\ref{prop:perf-coh}.
\end{proof}

\begin{rem}
The proof does not use flatness of \(\cinf_X\) over \(\CO_V\).  That
additional assertion belonged to a direct strictification argument and is
unnecessary in the sheaf-dga comparison used here.
\end{rem}

\section{Descent for the superconnection model}\label{sec:descent}

For an open set \(U\subset X\), write
\[
\CA^\bullet(U)=\Gamma(U,\CA_X^\bullet).
\]
Let \(\mathfrak U=\{U_i\}_{i\in I}\) be a locally finite open cover and set
\(U_{i_0\cdots i_p}=U_{i_0}\cap\cdots\cap U_{i_p}\).
Following \cite[Definition~3.1]{Wei2023}, let
\[
\operatorname{Tw}_{\mathfrak U}(\mathcal P_{\CA})
\]
be the dg category of twisted cohesive objects subordinate to
\(\mathfrak U\).  Thus an object consists of local cohesive modules \(E_i\)
and higher transition morphisms
\[
a^{p,1-p}_{i_0\cdots i_p}:
E_{i_p}|_{U_{i_0\cdots i_p}}
\longrightarrow
E_{i_0}|_{U_{i_0\cdots i_p}}
\]
satisfying the twisted Maurer--Cartan equation
\(\delta a+a\cdot a=0\), with \(a^{0,1}\) reproducing the given
differentials on the local morphism complexes, and with each
\(a^{1,0}_{ii}\) required to be invertible up to homotopy.

The full dg subcategory
\[
\operatorname{Tw}^B_{\mathfrak U}(\mathcal P_{\CA})
\subset
\operatorname{Tw}_{\mathfrak U}(\mathcal P_{\CA})
\]
of \cite[Definition~3.5]{Wei2023} consists of those objects for which there
are integers \(r\leq s\) and
\(N\geq0\), independent of \(i\), such that \(E_i^q=0\) outside
\([r,s]\) and every fiber of \(E_i^q\) has dimension at most \(N\).
This fiberwise formulation allows the rank to vary between connected
components.

\begin{thm}[Globally bounded superconnection descent]
\label{thm:eis-cohesive-descent}
Restriction induces a dg quasi-equivalence
\[
\mathcal T_{\mathfrak U}:
\mathcal P_{\CA^\bullet(X)}
\xrightarrow{\ \simeq_{\mathrm{qe}}\ }
\operatorname{Tw}^B_{\mathfrak U}(\mathcal P_{\CA}).
\]
If \(\mathfrak U\) is finite, then
\[
\operatorname{Tw}^B_{\mathfrak U}(\mathcal P_{\CA})
=
\operatorname{Tw}_{\mathfrak U}(\mathcal P_{\CA}),
\]
and hence
\[
\mathcal P_{\CA^\bullet(X)}
\simeq_{\mathrm{qe}}
\operatorname*{holim}_{[p]\in\Delta}
\prod_{(i_0,\ldots,i_p)\in I^{p+1}}
\mathcal P_{\CA^\bullet(U_{i_0\cdots i_p})}.
\]
Empty intersections are omitted.  The same conclusion holds for a locally
finite cover of compact \(X\), since such a cover has only finitely many
nonempty members.
\end{thm}

\begin{proof}
We verify that the proof of \cite[Theorem~5.5]{Wei2023} applies to
\(\CA_X^\bullet\).  Its degree-zero sheaf is \(\CA_X^0=\cinf_X\).
For every open \(U\), the bundle
\(\wedge^qV^\vee|_U\) satisfies
Lemma~\ref{lem:finite-smooth-generators}.  Therefore
\[
\CA^q(U)=\Gamma(U,\wedge^qV^\vee)
\]
is a finitely generated projective \(\cinf(U)\)-module.  Since only finitely
many \(q\) occur, the graded algebra \(\CA^\#(U)\) is flat over
\(\CA^0(U)\).

For an inclusion \(W\subset U\), multiplication gives an isomorphism of
graded \(\cinf(W)\)-modules
\[
\cinf(W)\otimes_{\cinf(U)}\CA^\#(U)
\xrightarrow{\ \cong\ }
\CA^\#(W).
\]
Indeed, choose an idempotent presentation of each finite projective module
\(\CA^q(U)\) and restrict its idempotent matrix to \(W\).  This is a graded
base-change statement, not a tensor-product description of the
differential.  For a homogeneous element \(e\) of a right cohesive module,
the pullback superconnection is
\[
\mathbb E_W(e\otimes f)
=
\mathbb E(e)|_W f+(-1)^{|e|}e\otimes d_Vf.
\]
Thus the restriction and pushforward constructions of
\cite[Section~2.2 and Proposition~2.8]{Wei2023} apply.

The definitions of twisted objects, their total morphism complexes, and the
adjunction between restriction and the associated global quasi-cohesive
module are formal for a sheaf of dgas; see
\cite[Sections~3.1--3.4]{Wei2023}.  The gluing of the underlying complexes
uses only \(\CA_X^0=\cinf_X\), locally finite partitions of unity, and the
common amplitude and rank bounds.  Consequently
\cite[Propositions~4.2 and 4.3]{Wei2023} give a bounded complex of finitely
generated projective \(\cinf(X)\)-modules quasi-isomorphic to the
degree-zero complex of the associated global quasi-cohesive module.

This verifies both inputs of Block's quasi-representability theorem: the
finite projective replacement just constructed and the flatness of
\(\CA^\#(X)\) over \(\CA^0(X)\).  Hence
\cite[Theorem~3.13]{Block2010} produces a global cohesive representative.
The formal argument of \cite[Propositions~5.2 and 5.3, Lemma~5.4]{Wei2023}
then proves that \(\mathcal T_{\mathfrak U}\) is a dg quasi-equivalence.

For a finite cover, the amplitudes and fiber ranks of the finitely many local
objects have common maxima, so the bounded and unrestricted twisted
categories agree.  The homotopy-limit description is
\cite[Theorem~3.1]{Wei2023}.
\end{proof}

\begin{rem}
The proof of Theorem~\ref{thm:eis-cohesive-descent} is a direct specialization
of Wei's descent mechanism, not a new descent formalism.  It does not use
ellipticity, Theorem~\ref{thm:local-mc}, or the CHL comparison; it applies to
any finite-rank involutive complex subbundle for which the displayed smooth
finite-type hypotheses hold.
\end{rem}

\begin{example}[Why the global bounds are necessary]
\label{ex:unbounded-descent}
Let \(X=\mathbb N\) be the discrete zero-dimensional manifold, take \(V=0\),
and cover \(X\) by \(U_m=\{m\}\).  Then
\[
\operatorname{Tw}_{\mathfrak U}(\mathcal P_{\CA})
\simeq
\prod_{m\geq1}\mathcal P_{\C}.
\]
The family \(E_m=\C^m\), concentrated in degree zero, is an object of the
right-hand side.  It cannot be the restriction of a global cohesive module.
Indeed, if \(R=\cinf(X)=\prod_{m\geq1}\C\), every term of a bounded complex
of finitely generated projective \(R\)-modules is a direct summand of
\(R^{N_q}\) for some \(N_q\).  Its fiber cohomology therefore has a uniform
finite dimension bound, whereas \(\dim H^0(E_m)=m\).  Thus the unrestricted
homotopy-limit statement is false for genuinely infinite locally finite
covers.  This deliberately disconnected example isolates the obstruction
from uniform finite generation; it does not assert the same obstruction for
every connected manifold.  It explains the bounded target already present
in Wei's theorem rather than strengthening or correcting that theorem.
\end{example}

\section{The coherent heart and transverse monodromy}
\label{sec:coherent-monodromy}

Choose a countable locally finite adapted cover
\[
U_i=B_i^d\times\Delta_i
\]
and let \(T_i=\{0\}\times\Delta_i\subset U_i\) denote its distinguished
transverse slice.  The disjoint union
\[
\mathcal T=\coprod_iT_i
\]
is a complete transverse atlas for \(D_{\R}\).  Such a cover exists by
paracompactness and second countability; countability ensures that
\(\mathcal T\) is again a second-countable complex manifold.  Let
\[
G_{\mathrm{mon}}
=
\operatorname{Mon}_{\mathcal T}(D_{\R})
\rightrightarrows\mathcal T,
\qquad
G_{\mathrm{hol}}
=
\operatorname{Hol}_{\mathcal T}(D_{\R})
\rightrightarrows\mathcal T
\]
be the restrictions to \(\mathcal T\) of the ordinary monodromy and
holonomy groupoids.  Thus monodromy arrows are leafwise paths modulo
leafwise homotopy relative to the endpoints, whereas holonomy arrows are
obtained by identifying paths with the same transverse germ.

The smooth groupoids integrating a regular foliation and their reduction to
complete transversals are reviewed in
\cite[Section~1, Theorem~1 and Proposition~1]{CrainicMoerdijk2001}; see also
\cite{MoerdijkMrcun2003} for the general foliation-groupoid framework.
The transverse holomorphic structure makes both restricted groupoids
possibly non-Hausdorff complex-analytic etale groupoids: in a local
bisection through an arrow, the target map is precisely its holomorphic
transverse transport germ.  In this section, a complex-analytic etale
groupoid is allowed to have a non-Hausdorff arrow manifold; its object
manifold is \(\mathcal T\), and its source and target maps are local
biholomorphisms.

\begin{defn}
For such a groupoid \(G\rightrightarrows G_0\), let
\[
\operatorname{Coh}_{\mathrm{an}}(G)
\]
be the category whose objects are coherent \(\CO_{G_0}\)-modules
\(\mathcal F\) equipped with an isomorphism
\[
\alpha:s^*\mathcal F\xrightarrow{\ \cong\ }t^*\mathcal F
\]
of coherent \(\CO_{G_1}\)-modules satisfying the unit and multiplication
identities.  Morphisms are coherent analytic morphisms compatible with
\(\alpha\).
\end{defn}

\begin{thm}[The coherent heart and monodromy]
\label{thm:coherent-monodromy}
Restriction to \(\mathcal T\), together with leafwise continuation, induces
an equivalence of abelian categories
\[
\operatorname{Coh}(X,\CO_V)
\xrightarrow{\ \sim\ }
\operatorname{Coh}_{\mathrm{an}}(G_{\mathrm{mon}}).
\]
Replacing the complete transverse atlas gives a naturally equivalent
category and equivalence.
\end{thm}

\begin{proof}
We first record the analytic refinement of the usual monodromy-sheaf
correspondence.  On a sufficiently small adapted box
\(U=B\times T\), write \(\pi:U\to T\).  By
Proposition~\ref{prop:local-ring},
\[
\CO_V|_U\cong\pi^{-1}\CO_T.
\]
If \(\mathcal F\) is a coherent \(\CO_V\)-module, then, after shrinking to a
smaller product box, it has a finite presentation
\[
(\pi^{-1}\CO_T)^p
\longrightarrow
(\pi^{-1}\CO_T)^q
\longrightarrow
\mathcal F
\longrightarrow0.
\]
The finitely many entries of the first map come, after one further
shrinking, from holomorphic functions on \(T\).  Taking their cokernel on
\(T\) produces a coherent analytic sheaf \(\mathcal G\) such that
\[
\mathcal F|_U\cong\pi^{-1}\mathcal G.
\]
The same finite-presentation argument applies to morphisms.  Hence coherent
\(\CO_V\)-modules and their morphisms are locally constant along plaques and
coherent analytic in the transverse direction.

Kock and Moerdijk prove that sheaves locally constant along the plaques are
equivalently equivariant sheaves for the full monodromy groupoid
\cite[Theorems~2.8 and 3.6]{KockMoerdijk1996}.  Their local connectedness,
openness, and local simple-connectivity hypotheses hold here because the
plaques in an adapted basis are balls.  Restriction to a complete
transversal gives an etale Morita presentation
\cite[Theorem~4.1 and Corollary~4.2]{KockMoerdijk1996}; compare
\cite[Section~1]{CrainicMoerdijk2001}.  Under the resulting equivalence the
ring sheaf \(\CO_V\) corresponds to \(\CO_{\mathcal T}\): on adapted boxes
this is the displayed identification
\(\CO_V=\pi^{-1}\CO_T\), and all transverse transition germs are
holomorphic.  The equivalence of sheaf categories therefore refines to an
equivalence of module categories over these ring objects.  The local
finite-presentation argument above shows that it restricts in both
directions to coherent modules, giving the asserted equivalence.

Restrictions of a foliation groupoid to any two complete transversals are
Morita equivalent.  Passing to their disjoint union gives a common
transverse refinement; its connecting arrows form the usual principal
bibundle.  Locally its structure maps are biholomorphisms, so analytic
coherence is preserved.  The resulting equivalence is therefore independent
of \(\mathcal T\).
\end{proof}

\begin{rem}[Ringed transverse monodromy stack]
\label{rem:ringed-monodromy-stack}
Let
\[
\mathfrak T_{\mathrm{mon}}
=
[\mathcal T/G_{\mathrm{mon}}]
\]
be the analytic stack presented by the restricted monodromy groupoid,
equipped with the ring object induced by \(\CO_{\mathcal T}\).  Theorem
\ref{thm:coherent-monodromy} can be written invariantly as
\[
\operatorname{Coh}(X,\CO_V)
\simeq
\operatorname{Coh}(\mathfrak T_{\mathrm{mon}}).
\]
Changing the complete transversal changes the etale presentation but not
this ringed stack or its coherent category.  This formulation of the
abelian-heart theorem does not by itself give an equivalence between
\(D^b_{\mathrm{coh}}(X,\CO_V)\) and perfect complexes on the ordinary
monodromy stack.  Such an equivalence is false for a general EIS, as
Example~\ref{ex:ordinary-groupoid-derived-failure} shows.  It does hold for
the holomorphic suspensions of
Theorem~\ref{thm:suspension-derived-monodromy}, where invariants are taken
homotopically rather than strictly.
\end{rem}

Let
\[
q:G_{\mathrm{mon}}\longrightarrow G_{\mathrm{hol}}
\]
be the monodromy-to-holonomy quotient, and let \(K=\ker(q)\) be the isotropy
subgroupoid of monodromy loops with trivial transverse holonomy germ.

\begin{cor}[The holonomy-basic coherent heart]
\label{cor:coherent-holonomy-basic}
Pullback is fully faithful and induces an equivalence
\[
q^*:
\operatorname{Coh}_{\mathrm{an}}(G_{\mathrm{hol}})
\xrightarrow{\ \sim\ }
\operatorname{Coh}_{\mathrm{an}}(G_{\mathrm{mon}})^{K=1},
\]
where the category on the right consists of objects for which every
\(k\in K\) acts strictly as the identity.
\end{cor}

\begin{proof}
A pulled-back action has the form
\(\alpha_g=\beta_{q(g)}\), so every kernel arrow acts as the identity.
Since \(q\) is surjective on arrows, a morphism between two pulled-back
objects is monodromy-equivariant if and only if it is
holonomy-equivariant; hence \(q^*\) is fully faithful.

Conversely, suppose \((\mathcal F,\alpha)\) is \(K\)-trivial.  For a
holonomy arrow \(h\), choose a monodromy lift \(g\) and set
\(\beta_h=\alpha_g\).  Two lifts differ by a kernel arrow, so this is
independent of the lift.  The unit and multiplication identities descend
from those for \(\alpha\), and local holomorphicity follows from the etale
quotient charts.  Thus \((\mathcal F,\alpha)\) descends to
\(G_{\mathrm{hol}}\).
\end{proof}

\begin{example}[Why the heart statement is not derived]
\label{ex:ordinary-groupoid-derived-failure}
Take the de Rham EIS \(V=T_{\C}X\).

If \(X=S^1\), then \(G_{\mathrm{mon}}\) restricted to a point has isotropy
\(\mathbb Z\), whereas \(G_{\mathrm{hol}}\) is the unit groupoid.  A
nontrivial character \(\mathbb Z\to\C^\times\) gives a coherent local system
which is not holonomy-basic.  Even for the trivial representation,
\[
\ext^1_{\C[\mathbb Z]}(\C,\C)\cong\C,
\qquad
\ext^1_{\C}(\C,\C)=0.
\]
Thus the derived analogue of \(q^*\) is not fully faithful.

If \(X=S^2\), then \(G_{\mathrm{mon}}\) and \(G_{\mathrm{hol}}\), restricted
to a point, are both trivial because \(S^2\) is simply connected.  However,
the endomorphism complex of the trivial cohesive object is
\(\Omega^\bullet(S^2;\C)\), and
\[
H^2\Omega^\bullet(S^2;\C)\cong\C.
\]
Consequently the full cohesive category is not the category of perfect
complexes on the ordinary monodromy or holonomy \(1\)-stack.  Ordinary
kernel-triviality is a complete criterion only at the coherent-heart level;
a derived criterion would require coherent higher descent data.

The obstruction here is the higher homotopy of the leaf.  By contrast, the
one-dimensional leaves of a holomorphic suspension are aspherical, and
Theorem~\ref{thm:suspension-derived-monodromy} shows that their ordinary
monodromy action recovers the full Morita-localized category when its
invariants are taken homotopically.
\end{example}

\begin{rem}[The next derived comparison]
\label{rem:transverse-holomorphic-rh}
The higher de Rham integration theorems of Block--Smith and Arias
Abad--Sch\"atz \cite{BlockSmith2014,AriasAbadSchaetz2013}, together with the
regular-foliation extension in \cite{ZengHigherRH}, suggest a genuinely
EIS-specific target built from leafwise infinity-local systems with
transverse Dolbeault coefficients.  The cited companion treats leafwise de
Rham coefficients; it does not prove the coefficient-enriched local theorem
needed here.  That theorem would have to construct an integration functor
\[
\mathcal P_{\Omega^\bullet(B)\widehat\otimes
\CA^{0,\bullet}(\Delta)}
\longrightarrow
\operatorname{Loc}_{\infty}
\left(
\Pi_\infty(B);
\mathcal P_{\CA^{0,\bullet}(\Delta)}
\right)
\]
and prove compatibility with \(\delb\), the mixed superconnection
components, sheafified morphism complexes, and holomorphic changes of
transverse charts.  The local Maurer--Cartan theorem and
Theorem~\ref{thm:eis-cohesive-descent} do not by themselves supply this
integration result.  Apart from the suspension case proved in
Section~\ref{sec:derived-suspensions}, no derived monodromy or
higher-holonomy equivalence is claimed here.
\end{rem}

\section{Basic cases}

\begin{example}[Point case]
If \(X=\{*\}\) and \(V=0\), then
\(\CA^\bullet=\CO_V=\C\).  A cohesive module is a bounded complex of
finite-dimensional vector spaces, and Theorem~\ref{thm:main} is the
tautological equivalence
\[
H^0(\mathcal P_{\C})\simeq D^b_{\mathrm{fd}}(\C).
\]
\end{example}

\begin{example}[Dolbeault case]
If \(X\) is a complex manifold and \(V=T^{0,1}X\), then
\(\CO_V=\CO_X\) and \(\CA_X^\bullet=\CA_X^{0,\bullet}\).
Theorem~\ref{thm:main} is Block's equivalence
\cite[Theorem~4.3]{Block2010}.
\end{example}

\begin{example}[de Rham case]
If \(V=T_{\C}X\), then \(\CO_V\) is the locally constant sheaf \(\C_X\) and
\(\CA_X^\bullet\) is the complexified de Rham dga.  Theorem~\ref{thm:main}
identifies cohesive modules with perfect complexes of finite local systems,
in agreement with \cite[Theorem~8.1]{CHL2021}.
\end{example}

\begin{example}[Product case]
Let \(Y\) be a compact smooth manifold and \(Z\) a compact complex manifold.
On \(X=Y\times Z\), set
\[
V=T_{\C}Y\oplus T^{0,1}Z.
\]
Then
\[
\CA_X^\bullet
\simeq
\Omega^\bullet_Y\widehat\otimes\CA_Z^{0,\bullet},
\qquad
\CO_V\simeq\C_Y\boxtimes\CO_Z.
\]
Let \(\mathcal L,\mathcal M\) be finite-dimensional local systems on \(Y\)
and let \(E,F\) be holomorphic vector bundles on \(Z\).  Their external
products are locally free \(\CO_V\)-modules, and
\begin{align*}
\rhom_{\CO_V}(\mathcal L\boxtimes E,\mathcal M\boxtimes F)
&\simeq
R\Gamma(Y,\mathcal L^\vee\otimes\mathcal M)
\otimes_{\C}^{\mathbf L}
R\Gamma(Z,E^\vee\otimes F),\\
\ext^k_{\CO_V}(\mathcal L\boxtimes E,\mathcal M\boxtimes F)
&\simeq
\bigoplus_{p+q=k}
H^p(Y,\mathcal L^\vee\otimes\mathcal M)
\otimes_{\C}H^q(Z,E^\vee\otimes F).
\end{align*}
Indeed, because the source is locally free, the mixed fine resolution
computes the derived Hom by the total complex
\[
\Omega^\bullet(Y;\mathcal L^\vee\otimes\mathcal M)
\widehat\otimes
\CA^{0,\bullet}(Z;E^\vee\otimes F).
\]
The de Rham and Dolbeault Hodge decompositions split both factors into their
finite-dimensional cohomology and contractible summands, which gives the
displayed K\"unneth formula.

For a concrete mixed computation, take \(Y=S^1\), \(Z=\mathbb P^1\), both
local systems trivial, \(E=\CO_{\mathbb P^1}\), and
\(F=\CO_{\mathbb P^1}(-2)\).  Since
\[
H^\bullet(S^1,\C)=(\C,\C),
\qquad
H^q(\mathbb P^1,\CO(-2))=
\begin{cases}
\C,&q=1,\\
0,&q\ne1,
\end{cases}
\]
we obtain
\[
\ext^k_{\CO_V}
(\C_{S^1}\boxtimes\CO,\C_{S^1}\boxtimes\CO(-2))
\simeq
\begin{cases}
\C,&k=1,2,\\
0,&\text{otherwise}.
\end{cases}
\]
The degree-two class uses simultaneously the de Rham direction of \(S^1\)
and the Dolbeault direction of \(\mathbb P^1\).
\end{example}

\section{Derived monodromy for holomorphic suspensions}
\label{sec:derived-suspensions}

Let \(Z\) be a compact complex manifold, let
\(\phi:Z\to Z\) be a biholomorphism, and form the mapping torus
\[
X_\phi
=
(\mathbb R\times Z)/
\bigl((t+1,z)\sim(t,\phi(z))\bigr).
\]
Equivalently, the deck generator is
\[
\gamma(t,z)=(t-1,\phi(z)).
\]
The invariant structure
\[
\widetilde V
=
T_{\C}\mathbb R\oplus T^{0,1}Z
\subset
T_{\C}(\mathbb R\times Z)
\]
descends to an EIS \(V\) on \(X_\phi\).  Its real foliation is generated
by the image of \(\partial_t\), and \(Z=\{0\}\times Z\) is a complete
transversal.  Thus \(X_\phi\) is a smooth, generally odd-dimensional
mapping torus; ``holomorphic'' refers to its transverse structure.  In this
section \(\CA^\bullet(X_\phi)\) always denotes the dga of this suspended
EIS, whose pullback to \(\R\times Z\) is
\[
\Omega^\bullet(\R)\widehat\otimes\CA^{0,\bullet}(Z),
\qquad
d_V=d_{\R}+\delb,
\]
not the ordinary de Rham dga of \(X_\phi\).

The restricted monodromy groupoid is the action groupoid
\[
G_{\mathrm{mon}}\cong\mathbb Z\ltimes_\phi Z.
\]
Here the positive generator acts by the first-return map \(\phi\).
Consequently, Theorem~\ref{thm:coherent-monodromy} gives
\[
\operatorname{Coh}(X_\phi,\CO_V)
\simeq
\operatorname{Coh}(Z)^\phi,
\]
where an object on the right is a coherent analytic sheaf
\(\mathcal F\) together with an isomorphism
\(\phi^*\mathcal F\cong\mathcal F\).  Ordinary holonomy replaces the
action groupoid by its effectivization, and
Corollary~\ref{cor:coherent-holonomy-basic} imposes strict triviality of
the ineffective isotropy.

\begin{lem}[Contractible real factor]
\label{lem:contractible-real-factor}
Let \(I\subset\R\) be an open interval and let
\(p:I\times Z\to Z\) be projection.  Pullback induces a dg
quasi-equivalence
\[
\mathcal P_{\CA^{0,\bullet}(Z)}
\xrightarrow{\ \simeq_{\mathrm{qe}}\ }
\mathcal P_{\Omega^\bullet(I)\widehat\otimes
\CA^{0,\bullet}(Z)}.
\]
\end{lem}

\begin{proof}
Choose \(t_0\in I\) and let \(j(z)=(t_0,z)\).  For two objects
\(E,F\in\mathcal P_{\CA^{0,\bullet}(Z)}\), the morphism complex between
their pullbacks is
\[
\Omega^\bullet(I)\widehat\otimes
\operatorname{Hom}_{\mathcal P_{\CA^{0,\bullet}(Z)}}(E,F).
\]
Evaluation at \(t_0\), together with the standard integration operator on
the interval, contracts this complex onto
\(\operatorname{Hom}_{\mathcal P_{\CA^{0,\bullet}(Z)}}(E,F)\).
Hence \(p^*\) is
quasi-fully faithful.

For essential surjectivity, let \((E,\mathbb E)\) be cohesive on
\(I\times Z\).  Its underlying bounded graded vector bundle is isomorphic
to the pullback of \(j^*E\), because \(I\) is contractible.  In such a
trivialization write
\[
\mathbb E=d_I+A(t)+dt\,B(t),
\]
where \(A(t)\) has total degree one and \(B(t)\) has total degree zero in
the completed endomorphism dga over \(Z\).  The equation
\(\mathbb E^2=0\) says that every \(A(t)\) is Maurer--Cartan and that its
\(t\)-derivative is the gauge derivative generated by \(B(t)\).  The
parallel-transport equation
\[
\partial_tg(t)=-B(t)g(t),
\qquad
g(t_0)=1,
\]
has an invertible solution given by the path-ordered exponential; its
convergence in the dg Arens--Michael algebra is the argument of
\cite[Theorem~4.4]{CHL2021}.  Flatness then makes \(g\) a closed invertible
morphism from \(p^*j^*(E,\mathbb E)\) to \((E,\mathbb E)\).  Thus \(p^*\)
is quasi-essentially surjective and hence a dg quasi-equivalence.
\end{proof}

We use the same symbol for a pretriangulated dg category and its dg nerve.
The cohesive categories below are stable and idempotent-complete:
pretriangulatedness is built into the cohesive construction, and
idempotents split by \cite[Lemma~3.3]{CHL2021}, since the dgas are
nonnegatively graded and flat over degree zero.  We use
\(\operatorname{Cat}^{\mathrm{perf}}_{\infty,\C}\) for the infinity-category
of small \(\C\)-linear stable idempotent-complete infinity-categories and
exact functors; compare \cite[Section~3]{BlumbergGepnerTabuada2013}.  The
forgetful functor to small \(\C\)-linear infinity-categories creates the
pullbacks of exact functors used below.  Indeed, finite limits and colimits
in such a pullback are computed componentwise, and compatible idempotents
split because the space of splittings of a split idempotent is contractible.
Thus the pullback is again stable and idempotent-complete and computes the
corresponding Morita pullback.

Write
\[
\mathcal C_Z=\mathcal P_{\CA^{0,\bullet}(Z)}
\]
and let \(\phi^*:\mathcal C_Z\to\mathcal C_Z\) be pullback.  Define the
homotopy fixed-point dg category by
\[
\mathcal C_Z^{h\phi}
=
\operatorname*{holim}\left(
\mathcal C_Z
\xrightarrow{(1,1)}
\mathcal C_Z\times\mathcal C_Z
\xleftarrow{(1,\phi^*)}
\mathcal C_Z
\right)
\]
in \(\operatorname{Cat}^{\mathrm{perf}}_{\infty,\C}\).  It is the derived
enhancement of the usual equivariant category: an object is represented by
\(E\in\mathcal C_Z\) and a homotopy-coherent equivalence
\(\phi^*E\simeq E\), while its coherent heart consists of sheaves equipped
with an isomorphism \(\phi^*\mathcal F\cong\mathcal F\).

\begin{thm}[Derived monodromy for a holomorphic suspension]
\label{thm:suspension-derived-monodromy}
Restriction to the transversal \(Z\), together with continuation once
around the base circle, induces an equivalence
\[
\boxed{
\mathcal P_{\CA^\bullet(X_\phi)}
\simeq
\mathcal C_Z^{h\phi}
}
\]
in \(\operatorname{Cat}^{\mathrm{perf}}_{\infty,\C}\), equivalently a
Morita equivalence of dg categories.
\end{thm}

\begin{proof}
Cover \(S^1\) by two open arcs \(U_0,U_1\) such that
\[
U_0\cap U_1=J_0\amalg J_1
\]
is the disjoint union of two open intervals.  Choose lifts of the arcs to
\(\R\).  They trivialize the inverse images as \(U_i\times Z\); on one
component of the overlap the transition is the identity of \(Z\), and on
the other it is \(\phi\).

The binary-cover form of
Theorem~\ref{thm:eis-cohesive-descent} gives
\[
\mathcal P_{\CA^\bullet(X_\phi)}
\simeq
\mathcal P_{\CA^\bullet(U_0\times Z)}
\mathop{\times}\limits^h_{
\mathcal P_{\CA^\bullet((J_0\amalg J_1)\times Z)}
}
\mathcal P_{\CA^\bullet(U_1\times Z)}.
\]
All three local categories and the global cohesive category are stable and
idempotent-complete, as noted above.  The displayed dg homotopy pullback
therefore also computes the pullback in
\(\operatorname{Cat}^{\mathrm{perf}}_{\infty,\C}\); this uses closure of
stable idempotent-complete categories under limits, not preservation of
pullbacks by Morita localization.

The category on the disconnected overlap is the product of the categories
on its two components.  Indeed, the central idempotents of the product dga
split every finite projective module and its superconnection; conversely,
two cohesive modules combine over the product dga.

Lemma~\ref{lem:contractible-real-factor} replaces the three vertices of this
cospan by
\[
\mathcal C_Z,\qquad
\mathcal C_Z\times\mathcal C_Z,\qquad
\mathcal C_Z.
\]
Fix one pair of evaluation functors on \(J_0\amalg J_1\).  On the first
leg, evaluation after restriction is smoothly homotopic to evaluation on
\(U_0\) in each component.  On the second leg the same statement holds,
with \(\phi^*\) retained on the component where the trivializations differ.
More explicitly, every evaluation functor is an inverse to the same
constant-extension quasi-equivalence of
Lemma~\ref{lem:contractible-real-factor}.  In an infinity-category the space
of inverses to an equivalence is contractible, so these evaluation functors
are naturally equivalent.  The two resulting squares share the chosen
middle evaluation functor.  Since the indexing cospan has no composable
nonidentity arrows, no further coherence is required.  The resulting
cospan is
\[
\mathcal C_Z
\xrightarrow{(1,1)}
\mathcal C_Z\times\mathcal C_Z
\xleftarrow{(1,\phi^*)}
\mathcal C_Z.
\]
Homotopy limits preserve equivalences of diagrams, and its homotopy
pullback is \(\mathcal C_Z^{h\phi}\) by definition.
\end{proof}

\begin{rem}
If \(Z=\{*\}\) and \(\phi=1\), the theorem gives
\[
\mathcal P_{\Omega^\bullet(S^1)}
\simeq
\operatorname{Perf}(\C)^{h\mathbb Z}.
\]
For the trivial object, the right-hand mapping complex is
\(R\Gamma(\mathbb Z,\C)\), with cohomology in degrees zero and one.  Thus
the homotopy fixed points retain the degree-one class that a strict
equalizer would discard.
\end{rem}

\begin{cor}[Equivariant Ext spectral sequence]
\label{cor:suspension-equivariant-ext}
Let
\[
(\mathcal F,\theta_{\mathcal F}),
\quad
(\mathcal G,\theta_{\mathcal G})
\in\operatorname{Coh}(Z)^\phi,
\qquad
\theta_{\mathcal F}:\phi^*\mathcal F\xrightarrow{\sim}\mathcal F,
\quad
\theta_{\mathcal G}:\phi^*\mathcal G\xrightarrow{\sim}\mathcal G,
\]
and let \(\widetilde{\mathcal F},\widetilde{\mathcal G}\) be the
corresponding coherent \(\CO_V\)-modules, obtained by descending
\(\operatorname{pr}_Z^{-1}\mathcal F\) and
\(\operatorname{pr}_Z^{-1}\mathcal G\) using the displayed isomorphisms.
The induced action on
\(\ext_Z^q(\mathcal F,\mathcal G)\) is
\[
T(\alpha)
=
\theta_{\mathcal G}\circ\phi^*(\alpha)\circ
\theta_{\mathcal F}^{-1}.
\]
There is a spectral sequence
\[
E_2^{p,q}
=
H^p\left(
\mathbb Z,\ext_Z^q(\mathcal F,\mathcal G)
\right)
\Longrightarrow
\ext_{\CO_V}^{p+q}
\left(
\widetilde{\mathcal F},\widetilde{\mathcal G}
\right).
\]
Since \(\mathbb Z\) has cohomological dimension one, there are natural
short exact sequences
\[
\begin{split}
0\longrightarrow&
\operatorname{coker}\left(
T-1:\ext_Z^{k-1}(\mathcal F,\mathcal G)
\longrightarrow
\ext_Z^{k-1}(\mathcal F,\mathcal G)
\right)\\
\longrightarrow&
\ext_{\CO_V}^{k}
\left(
\widetilde{\mathcal F},\widetilde{\mathcal G}
\right)\\
\longrightarrow&
\ker\left(
T-1:\ext_Z^{k}(\mathcal F,\mathcal G)
\longrightarrow
\ext_Z^{k}(\mathcal F,\mathcal G)
\right)
\longrightarrow0.
\end{split}
\]
\end{cor}

\begin{proof}
Put \(M=\rhom_Z(\mathcal F,\mathcal G)\).  The equivariant structures
induce the chain automorphism \(T\) displayed above.  By
Theorem~\ref{thm:coherent-monodromy}, the strict equivariant sheaves descend
to \(\widetilde{\mathcal F}\) and \(\widetilde{\mathcal G}\).  The
equivalence in Theorem~\ref{thm:suspension-derived-monodromy} was constructed
from these same restriction and gluing functors, so it takes the two
equivariant Dolbeault objects to those descended modules.  The dg
enhancement of Remark~\ref{rem:dg-enhancement}, applied both to \(X_\phi\)
and to the Dolbeault EIS on \(Z\), identifies their mapping complexes with
the corresponding derived sheaf Hom complexes.

Mapping complexes in a homotopy limit of stable dg categories are the
corresponding homotopy limits of mapping complexes.  Therefore
\[
\rhom_{\CO_V}
\left(
\widetilde{\mathcal F},\widetilde{\mathcal G}
\right)
\simeq
R\Gamma(\mathbb Z,M)
\simeq
\operatorname{Cone}(T-1:M\longrightarrow M)[-1].
\]
The last model follows by applying
\(\rhom_{\C[\mathbb Z]}(-,M)\) to the length-one free resolution
\[
0\longrightarrow
\C[\mathbb Z]
\xrightarrow{g-1}
\C[\mathbb Z]
\longrightarrow
\C
\longrightarrow0.
\]
The cohomology spectral sequence of derived invariants is the displayed
group-cohomology spectral sequence, and the long exact sequence of the cone
gives the short exact sequences.
\end{proof}

For the canonical equivariant sheaf \(\CO_Z\),
Corollary~\ref{cor:suspension-equivariant-ext} becomes
\[
H^p\left(\mathbb Z,H^q(Z,\CO_Z)\right)
\Longrightarrow
\ext_{\CO_V}^{p+q}(\CO_V,\CO_V).
\]
This is also the Leray spectral sequence for the mapping-torus projection
\(f:X_\phi\to S^1\).  In particular, it gives
\[
\begin{split}
0\longrightarrow&
\operatorname{coker}\left(
\phi^*-1:H^{k-1}(Z,\CO_Z)\longrightarrow H^{k-1}(Z,\CO_Z)
\right)\\
\longrightarrow&
H^k(X_\phi,\CO_V)\\
\longrightarrow&
\ker\left(
\phi^*-1:H^k(Z,\CO_Z)\longrightarrow H^k(Z,\CO_Z)
\right)
\longrightarrow0,
\end{split}
\]
where
\[
\ext^k_{\CO_V}(\CO_V,\CO_V)
\cong H^k(X_\phi,\CO_V).
\]

\begin{example}[An elliptic-curve suspension]
\label{ex:elliptic-curve-suspension}
Take \(Z=E\) to be an elliptic curve and \(\phi=[-1]\).  The induced action
is \(+1\) on \(H^0(E,\CO_E)\cong\C\) and \(-1\) on
\(H^1(E,\CO_E)\cong\C\).  Hence
\[
\ext^k_{\CO_V}(\CO_V,\CO_V)
\cong
\begin{cases}
\C,&k=0,1,\\
0,&\text{otherwise}.
\end{cases}
\]
By contrast, the product EIS \(S^1\times E\) has self-Ext dimensions
\(1,2,1\) in degrees \(0,1,2\).  The suspension monodromy therefore removes
the transverse \(H^1(E,\CO_E)\) invariant and coinvariant contributions.
Moreover, for \(\phi=[-1]\), ordinary holonomy replaces
\(\mathbb Z\ltimes E\) by \((\mathbb Z/2)\ltimes E\); the surviving
degree-one class is another concrete reason that holonomy effectivization
does not model the full derived category, whereas
\(\mathcal C_E^{h\phi}\) does.
\end{example}

\printbibliography

\end{document}